\documentclass[11pt]{amsart}

\usepackage[english]{babel}
\usepackage{amsmath,amssymb,amsthm}
\usepackage{graphicx}
\usepackage[colorlinks=true,linkcolor=blue,urlcolor=blue,citecolor=blue]{hyperref}

\newcommand{\eps}{\epsilon}

\newcommand{\F}{\mathcal{F}}
\newcommand{\C}{\mathbb{C}}
\newcommand{\R}{\mathbb{R}}
\newcommand{\Z}{\mathbb{Z}}
\newcommand{\bd}{\partial}

\newcommand{\prob}[1]{\mathbb{P} \left( #1 \right)}
\newcommand{\E}{\mathbf{E}}
\newcommand{\expect}[1]{\E \left[ #1 \right]}
\newcommand{\condexpect}[2]{\mathbf{E} \left[ \left. #1 \right| #2 \right]}
\newcommand{\indicate}[1]{\mathbf{1}  \left \{ #1 \right \}}
\newcommand{\condprob}[2]{\mathbb{P} \left( \left. #1 \right| #2 \right)}
\newcommand{\probb}[2]{\mathbb{P}^{#1} \left( #2 \right)}
\newcommand{\Es}{\operatorname{Es}}
\newcommand{\D}{\mathbb{D}}
\newcommand{\Lo}{\operatorname{L}}
\newcommand{\ball}[2]{B(#1;#2)}
\newcommand{\LE}[1]{\Lo(#1)}
\newcommand{\RLE}[1]{\overline\Lo(#1)}
\newcommand{\rev}[1]{\overline{#1}}
\newcommand{\Disk}{\D}
\newcommand{\abmetric}[2]{\rho(#1,#2)}
\newcommand{\abmetricnoarg}{\rho}
\newcommand{\supdist}{\operatorname{d}}
\newcommand{\levydist}{\operatorname{d}_{LP}}

\newtheorem{theorem}{Theorem}[section]
\newtheorem{lemma}[theorem]{Lemma}

\newtheorem{conjecture}[theorem]{Conjecture}

\begin{document}

\title[Partial results on convergence of LERW  to SLE(2) in natural parametrization]{Some partial results on the convergence of loop-erased random walk to SLE(2) in the natural parametrization}

\author{Tom Alberts}
\address{Department of Mathematics, California Institute of Technology, Pasadena, CA 91125, USA}
\email{alberts@caltech.edu}

\author{Michael J.~Kozdron}
\address{Department of Mathematics \& Statistics, University of Regina, Regina, SK S4S 0A2, Canada}
\email{kozdron@stat.math.uregina.ca}

\author{Robert Masson}
\address{Tradelink LLC, 71 South Wacker Drive, Suite 1900, Chicago, IL 60606, USA}

\begin{abstract}
We outline a strategy for showing convergence of loop-erased random walk on the $\Z^2$ square lattice to SLE(2), in the supremum norm topology that takes the time parametrization of the curves into account. The discrete curves are parametrized so that the walker moves at a constant speed determined by the lattice spacing, and the SLE(2) curve has the recently introduced \emph{natural time parametrization}. Our strategy can be seen as an extension of the one used by Lawler, Schramm, and Werner to prove convergence modulo time parametrization. The crucial extra step is showing that the expected occupation measure of the discrete curve, properly renormalized by the chosen time parametrization, converges to the occupation density of the SLE(2) curve, the so-called SLE Green's function.  Although we do not prove this convergence, we rigorously establish some partial results in this direction including a  new loop-erased random walk estimate.
\end{abstract}


\maketitle

\section{Introduction}

The Schramm-Loewner evolution (SLE) is a one-parameter family of random curves in two dimensions introduced by Schramm~\cite{Sch00} while studying possible scaling limits of loop-erased random walk. In this paper, Schramm proved that if the scaling limit exists, is conformally invariant, and satisfies a certain domain Markov property then the scaling limit must be SLE($2$), a random curve generated by the Loewner equation with a Brownian motion of variance $2$ as its driving function. The existence of this conformally invariant scaling limit was later confirmed by Lawler, Schramm and Werner~\cite{LSW04} who proved weak convergence of the properly rescaled discrete curves to  SLE$(2)$ as random curves modulo reparameterization.

The goal of the present work is to discuss a stronger convergence of loop-erased random walk to SLE$(2)$, namely convergence with respect to a topology that fully takes into account the time parametrization of the curves.   We outline a possible strategy for the proof  that  loop-erased random walk converges to  SLE$(2)$ in the natural time parametrization, and although we are unable to rigorously establish all of the steps in the proof, we believe that there is value in recording the partial results that we have obtained. In particular, we have been able to identify the specific loop-erased random walk estimates that are needed to carry out our strategy and deduce strong convergence. 

Consider a loop-erased random walk on $\Z^2$ started at the origin, and suppose that $M_n$ is the first time that it reaches the circle of radius $n$. Assume further that it moves at unit speed between sites. Let $X=X(t)$ be its time reversal so that $X(t)$, $0 \le t \le M_n$, is a continuous, piecewise-linear function from the circle of radius $n$ to the origin.  In order to prove that
loop-erased random walk converges to SLE$(2)$ in the natural time parametrization there are two things that need to be done, namely
\begin{description}
\item[(i)]  to show that  if $\sigma_n(t)$ is a suitable ``speed function'' (i.e., continuous and increasing), then
$$t \mapsto  Y_n(t) := \frac{1}{n} X(\sigma_n(t))$$
converges weakly as $n \to \infty$ with respect to the topology of the supremum norm on curves, and
\item[(ii)] to identify the limit as  SLE(2) in a particular parametrization.
\end{description}
Recent developments make both parts of this problem appear more tractable.  The first, which concerns (ii), is Lawler and Sheffield's identification~\cite{LS09} of the ``natural'' parametrization for the SLE(2) curve; see~\cite{LawRez,LawRez2,LawZhou} for extensions of this result.  Under Schramm's original definition of SLE the curves are parametrized so that their capacity  grows linearly. This is the best way to analyze the curves by way of the Loewner equation, but is not natural when one considers SLE as the scaling limit of a discrete model.

The second development, concerning (i), is Barlow and Masson's tightness bounds~\cite{BM09} on a certain rescaling of $M_n$. The ideal choice of speed function is $\sigma_n(t) = n^{5/4}t$ because of the long-standing physical prediction~\cite{duplantier,GutBur,Maj}  that the growth exponent for loop-erased random walk is $5/4$. But proving (i) for this speed function is difficult for the following reason: if $Y_n$ converges in law under the topology of the supremum norm (see Section~\ref{backgroundsect} for an exact definition) then so too does the \textit{lifetime} of the curve. For the choice $\sigma_n(t) = n^{5/4} t$ the lifetime is $n^{-5/4} M_n$, but it is not even known that this sequence of random variables is tight, let alone that it has a limiting distribution.   Showing such a result appears to be genuinely difficult. Even proving the exact asymptotic behaviour of $\expect{M_n}$ is beyond the reach of current methods. At present, the best known result is that
\begin{equation}\label{KenRes}
\lim_{n \to \infty} \frac{\log \expect{M_n}}{\log n} = \frac{5}{4}.
\end{equation}
This was originally proved by Kenyon~\cite{Ken00}, and later reproved by Masson~\cite{Mas09} using different methods. Subsequently, Barlow and Masson~\cite{BM09} proved that the family
$$\frac{M_n}{\expect{M_n}}$$
is tight (although they say nothing about the asymptotic behaviour of $\expect{M_n}$ directly). For us, this makes  $\sigma_n(t) = \expect{M_n}t$ an attractive choice of  speed function as it guarantees that the lifetime of $Y_n$ is tight. Of course, this alone is not sufficient to prove (i), but it is a suggestive starting point.

Kenyon's proof of~\eqref{KenRes} is based on estimates of the probability that loop-erased random walk goes through a particular edge; see also~\cite{KenWil} for related results. 
A recent paper by Lawler~\cite{Law13} has improved Kenyon's estimates by showing that for loop-erased random walk connecting opposite sides of an $n \times n$ square, there is probability $n^{-3/4}$, up to multiplicative constants, that a particular edge in the middle is used.  It is expected\footnote{Lawler, G.F.: Private communication, 2013.} (but not yet proved) that this result can be strengthened to show that  $\expect{M_n} \asymp n^{5/4}$, after some additional work of transferring the result from the ``chordal'' case to the ``radial'' case. Combined with the result of Barlow and Masson~\cite{BM09}, this would prove that $n^{-5/4}M_n$ is tight. 

The outline of the paper is as follows.  In Section~\ref{backgroundsect} we review the basics of the Schramm-Loewner evolution, including a careful statement that loop-erased random walk converges weakly to SLE(2)  in the topology of the supremum norm on curves modulo reparametrization. We then discuss the natural parametrization for SLE and
how the time parametrization can be encoded as an occupation measure with natural conformal covariance and domain Markov properties.
In Section~\ref{SectConvStrat},  we outline the steps in our strategy for showing convergence of loop-erased random walk to SLE$(2)$  with respect to the topology  of the supremum norm on curves. The basic idea is to show convergence of the random occupation measures defined by the discrete curves by proving that their law is tight and that any subsequential limit has the unique set of properties that characterize the SLE natural occupation measure. In this way, our strategy mimics the original proof of convergence modulo time parametrization. As mentioned, we are not currently able to rigorously complete all of the steps in our strategy. However, in Section~\ref{lerwsect} we are able to establish some partial results on convergence of the expected occupation measure for the discrete curves to the expected occupation measure for the SLE$(2)$ curve. In particular, Theorem~\ref{LERWupperboundThm} is a new estimate for loop-erased random walk that extends those found in~\cite{BM09}. Section~\ref{conclsect} contains a brief conclusion and suggestions for future work.

\section{Background}\label{backgroundsect}

We now introduce the notation and background material that will be necessary for this paper. A standard reference for loop-erased random walk is Chapter~7 of~\cite{LawlerGreen} and basic material about SLE may  be found in~\cite{SLEbook}. Suppose that $\C$ denotes the complex plane, $\D= \{z : |z|<1\}$ denotes the unit disk, and $B(z;\epsilon) = \{w \in \C : |z-w| < \epsilon\}$ denotes the ball of radius $\epsilon$ centred at $z$. A domain $D \subset \C$ is called a grid domain (with respect to $\Z^2$) if the boundary of $D$ consists of edges of the lattice $\Z^2$. If $A \subset \Z^2$, the boundary of $A$ is defined as $\bd A = \{ \text{$x \in \Z^2 \setminus A : |x-y|=1$ for some $y \in A$}\}$.

\subsection{Topologies on curves}\label{sect_topology}

Suppose that $\Omega$ denotes the space of all continuous curves $\gamma:[0,t_{\gamma}] \to \C$ where $t_{\gamma} \in [0,\infty)$ is the time duration or lifetime of $\gamma$. We use a slight variation of the usual distance function on $\Omega$ (the one induced by the supremum norm), defined by
\begin{align}\label{strongmetric}
\supdist(\gamma_1, \gamma_2) = |t_{\gamma_1} - t_{\gamma_2}| + \sup_{0 \leq s \leq t_{\gamma_1} \vee \, t_{\gamma_2}} |\gamma_1(s) - \gamma_2(s)|.
\end{align}
If the curves $\gamma_1$ and $\gamma_2$ have different time durations, then the curve with the smaller time duration sits at its endpoint until the curve with the larger time duration ends. Following~\cite{AB99}, we also define
\begin{equation}\label{abmetric}
\abmetric{\gamma_1}{\gamma_2} = \inf_{\phi}\left[ \sup_{0 \le s \le t_{\gamma_1}} | \gamma_1(s) - \gamma_2(\phi(s))|\right]
\end{equation}
where the infimum is over all increasing homeomorphisms $\phi:[0,t_{\gamma_1}] \to [0,t_{\gamma_2}]$. We will say that $\gamma'$ is a reparametrization of $\gamma \in \Omega$ and write $\gamma' \sim \gamma$ if $\abmetric{\gamma'}{\gamma} = 0$. Let $\tilde \gamma$ denote the equivalence class of $\gamma$ modulo reparametrization; that is,
$$\tilde \gamma = \{\gamma' \in \Omega : \rho(\gamma',\gamma)=0\},$$
and suppose $\tilde \Omega = \Omega/\sim$ denotes the set of equivalence classes of curves modulo reparametrization. It can be checked that $(\tilde \Omega, \abmetricnoarg)$ is a complete metric space; see Lemma~2.1 of~\cite{AB99}. Given an equivalence class $\tilde{\gamma}$ and a representative $\gamma \in \tilde{\gamma}$, we let $$\widetilde{\gamma[0,t]}$$ denote the equivalence class corresponding to the curve $s \mapsto \gamma(s)$, for $0 \leq s \leq t$. Let
$$\tilde{\F}_t(\gamma) := \sigma \left( \widetilde{\gamma[0,s]}, s \leq t \right).$$
Observe that if $\gamma$, $\eta \in \tilde{\gamma}$, then $\tilde{\F}_t(\gamma)$ is a time change of the filtration $\tilde{\F}_t(\eta)$.
Finally, to avoid introducing extra notation, we will write an (arbitrary) representative of the equivalence class $\tilde \gamma$ as $\tilde \gamma$ as well.

\subsection{Review of loop-erased random walk and radial SLE}

\subsubsection*{Loop-erased random walk}

A loop-erased random walk is a non-self intersecting path obtained by chronologically erasing loops from a simple random walk path. The following loop-erasing procedure, originally introduced by Lawler~\cite{Law80}, associates  a self-avoiding path to each finite simple random walk path in $\Z^2$. 

Let $S = S[0,m] = [S(0), S(1), \ldots, S(m)]$ be a simple random walk path of length $m$.
We construct $\LE{S}$, the  loop-erased part of $S$, recursively as follows. If $S$ is self-avoiding already,
set $\LE{S}=S$.  If not, set $s_0 = \max\{j : S(j)=S(0)\}$, and for $i > 0$,
 set $s_i = \max\{j : S(j) = S(s_{i-1}+1) \}$. If we set $k = \min\{i : s_i=m\}$, then
$\LE{S} = [S(s_0), S(s_1), \ldots, S(s_k)]$. Note that  $\LE{S}(0)=S(0)$ and $\LE{S}(s_k) = S(m)$. Also notice that the loop-erasing algorithm depends on the order of the points.
If $\omega=[\omega(0),\omega(1), \ldots,\omega(k)]$ is a lattice path, write $\rev{\omega} = [\omega(k),\omega(k-1), \ldots,\omega(0)]$ for its reversal. Thus, if we define reverse loop-erasing by $\RLE{S} = \rev{\LE{\rev{S}}}$, then one can construct a path $S$ such that $\LE{S} \neq \RLE{S}$.  However, it is a fact that both $\LE{S}$ and $\RLE{S}$ have the same distribution; see Lemma~3.1 of~\cite{LSW04}.  Consequently,  we will not be careful to distinguish between $\LE{\rev{S}}$ and $\RLE{S}$. In this paper, we will consider the loop-erasure of simple random walk started at $0$ and stopped when hitting the boundary of some fixed grid domain $D$. We call this loop-erased random walk in $D$.

Suppose that  $S$ is a simple random walk path started at the origin stopped when reaching the disk of radius $n$ so that $S = S[0,\tau_n]$ where $\tau_n = \min\{j : |S(j)| \ge n\}$. Let  $X= \RLE{S}$,  let $m_n = \min\{i : s_i = \tau_n\}$, and let $M_n = s_{m_n}$ so that $X$ is a loop-erased random walk on $\mathbb{Z}^2$  with $X(0) = S(\tau_n)$ and $X(M_n)=0$. We extend $X$ to a continuous function $X:[0,M_n] \to \C$ by linearly interpolating between steps at unit speed.  In the notation of Section~\ref{sect_topology},
$X \in \Omega$  with lifetime $t_X = M_n \in (0,\infty)$.

Loop-erased random walk has the important \emph{domain Markov property}; see Lemma~3.2 of~\cite{LSW04} and Lemma~3.2 of~\cite{BM09} for further discussion. Suppose $X=[X(0), X(1), \ldots, X(\ell)]$ is the loop-erasure of the time-reversal of a simple random walk that is started from $0$ and stopped when exiting $D$. Conditioned on the first $j$ steps of $X$, the distribution of the rest of the curve is the same as loop-erased random walk in $D \setminus X[0,j]$ from $X(j)$ to 0.

\subsubsection*{Radial SLE}

Assume that the unit disk $\Disk$ is slit by a non self-intersecting curve
$\gamma$ in such a way that $\Disk \setminus \gamma$ is simply
connected and contains $0$. We may then parameterize the curve by capacity; that is, we may parameterize $\gamma(t)$ such that
the normalized conformal map
$g_t: \Disk \setminus \gamma[0,t] \to \Disk$ satisfies
\begin{equation*}
g_t(z)=e^{t}z + O(z^2),
\end{equation*}
around the origin for each $t \ge 0$. A theorem due to Loewner states that
the \emph{Loewner chain} $(g_t)$, $t \ge 0$, satisfies the Loewner differential equation
\begin{equation}\label{LODE}
\partial_tg_t(z) = g_t(z) \frac{\xi(t)+g_t(z)}{\xi(t)-g_t(z)}, \quad g_0(z)=z,
\end{equation}
where $\xi(t)=g_t(\gamma(t))$ is a unique continuous unimodular
function.

On the other hand, consider a function that is continuous on $[0, \infty)$ with values in
$\partial \Disk$. The Loewner differential equation~\eqref{LODE} can then be solved up to time
$t$ for all $z$ outside $K_t=\{w : \tau(w) \le t\}$, where
$\tau(w)$ is the hitting time of $\xi(t)$ by $g_t(w)$; see~\cite{SLEbook} for precise definitions. Note that $g_t$ maps $\Disk \setminus
K_t$ conformally onto $\Disk$ for $t \ge 0$, and that $K_t$ is
called the hull of the Loewner chain. The function $\xi$ is called the driving function for the Loewner chain $(g_t)$.  If the limit
\begin{equation*}
\gamma(t)=\lim_{r \to 1-}g_t^{-1}(r\xi(t))
\end{equation*}
exists for $t > 0$ and $t \mapsto \gamma(t)$ is continuous, then $(g_t)$ is said to be generated by a curve, in which case the connected components of $\D \setminus \gamma[0,t]$ and $\D \setminus K_t$ that contain the origin are the same.

If  $B(t)$ is standard Brownian motion and $\kappa >0$ is a parameter, then the \emph{radial Schramm-Loewner
evolution with parameter $\kappa$}, abbreviated SLE$(\kappa)$, is obtained if $\xi(t)=\exp\{i B(\kappa t) \}$. It is known that SLE$(\kappa)$ is generated by a curve; see~\cite{LSW04} and~\cite{rohde_schramm}. Moreover, it is also known that if $\kappa \le 4$, then the curve is simple.

Our space $\Omega$ is defined as curves that have finite lifetimes, but the radial SLE curve in the capacity parametrization has infinite lifetime.   However, since $t \mapsto \gamma(t)$ is continuous for all $t \ge 0$ and 
$\gamma(t) \to  0$ as $t \to \infty$ (see~\cite{Law11}),  it is possible to find a reparametrization of $\gamma$ that has a finite lifetime. For instance, if $\gamma$ is a radial SLE curve in the capacity parametrization, then 
$$t \mapsto
\begin{cases}
\gamma(t/(1-t)), &0 \le t <1,\\
0, &t=1,
\end{cases}
$$
has lifetime 1.

\subsubsection*{Loop-erased random walk  converges to radial SLE(2)}

It was proved by Lawler, Schramm, and Werner~\cite{LSW04} that loop-erased random walk converges to SLE(2). We will now give a careful statement of their theorem.
Let $D \subsetneq \C$ be a simply connected domain with $0 \in D$, and let $\lambda_n$ denote the law of the time-reversal of the  loop-erasure of simple random walk on $n^{-1}\Z^2$, started at 0 and stopped when reaching $\bd D$. 
Let $\lambda$ denote the law of the conformal image of a radial SLE(2) started uniformly on $\bd \D$, where the conformal map is from $\D \to D$ and fixes the origin.

\begin{theorem}[\cite{LSW04}]\label{LSWthm}
The measures $\lambda_n$ converge weakly to $\lambda$ as $n \to \infty$ with respect to the metric $\rho$ on the space of curves given by~\eqref{abmetric}.
\end{theorem}

Recently,  a rate of convergence for the \emph{Loewner driving process} for loop-erased random walk  to  the driving process  for radial SLE$(2)$ was obtained by Bene\v{s}, Johansson Viklund, and Kozdron~\cite{BJK}.  A further extension of this work was provided by Johansson Viklund~\cite{FJVrate} who proved a  rate for the convergence in Theorem~\ref{LSWthm}.

We observe here the following  consequence of Theorem~\ref{LSWthm}. Let $\gamma$ be a radial SLE(2) started uniformly on $\bd \D$, and suppose that $X(t)$, $0\le t \le M_n$, denotes the time reversal of a  loop-erased random walk on $\Z^2$ started at the origin and stopped at $M_n$, the time the loop-erased random walk reaches the circle of radius $n$.
If $z \in \D$, $\epsilon>0$, and
$$Y_n(t) = \frac{1}{n}X(\sigma_n(t))$$
where $\sigma_n(t)$ is a speed function, then
\begin{equation}\label{LSWconseq}
\lim_{n \to \infty} \prob{\tilde Y_n \cap B(z;\epsilon) \neq \emptyset} = \prob{\tilde \gamma \cap B(z;\epsilon) \neq \emptyset}.
\end{equation}

\subsection{Encoding time parameterizations via occupation measures}

Consider a continuous, simple curve in $\C$. A general principle is that all the information in this curve may be encoded by the equivalence class it belongs to plus its \textit{occupation measure}. The occupation measure associates to each Borel subset of the plane the amount of time spent by the curve in that subset, i.e., given $\gamma$ its occupation measure $\nu_{\gamma}$ is
\begin{align*}
\nu_{\gamma}(A) := \int_0^{t_{\gamma}} \indicate{\gamma(s) \in A} \, ds.
\end{align*}
Clearly $\nu_{\gamma}$ is supported on the trace of $\gamma$, and its total mass is the lifetime of  $\gamma$.  It is easy to see that  $\gamma$ can be recovered from the pair $(\tilde{\gamma}, \nu_{\gamma})$. Indeed, given the pair let $\eta$ be any representative of $\tilde{\gamma}$, and define $\Theta_{\eta}(t) = \nu_{\gamma}(\eta[0,t])$. Then $\Theta_{\eta}$ is necessarily a continuous, non-decreasing function of $t$, and it is an easily verified but important fact that
\begin{align}\label{Tinverse_procedure}
\gamma(t) = \eta(\Theta_{\eta}^{-1}(t)),
\end{align}
independently of the choice of $\eta$ (we use the right-continuous inverse in the case that $\Theta_\eta$ is not strictly increasing).

To be precise, let $\mathcal{M}$ be the set of positive finite Borel measures on $\mathbb{D}$. Define $T : \Omega \to \tilde{\Omega} \times \mathcal{M}$ by
\begin{align*}
T : \gamma \mapsto (\tilde{\gamma}, \nu_{\gamma}),
\end{align*}
and let $S : \tilde{\Omega} \times \mathcal{M} \to \Omega$ be its inverse map as defined by the procedure in~\eqref{Tinverse_procedure}. Throughout we equip $\mathcal{M}$ with the topology of weak convergence of measures and $\tilde{\Omega} \times \mathcal{M}$ with the product topology. Observe that this topology is metrizable by 
\begin{equation}\label{productmetric}
\abmetric{\tilde{\gamma}_1}{\tilde{\gamma}_2} + \levydist(\nu_1, \nu_2),
\end{equation}
where $\levydist$ is the L\'{e}vy-Prokhorov metric on $\mathcal{M}$ defined by
\begin{align*}
&\levydist(\mu, \nu) \\
&\quad= \inf \left \{ \epsilon > 0 : \mu(A) \leq \nu(A^{\epsilon}) + \epsilon, \nu(A) \leq \mu(A^{\epsilon}) + \epsilon \textrm{ for all Borel } A \subset \C \right \}.
\end{align*}
Here $A^{\epsilon}$ is the set of all points within distance $\epsilon$ of $A$. Under these topologies the following is true.

\begin{lemma}\label{lemma:TScontinuity}
The mapping $T$ is continuous on $\Omega$, and the mapping $S$ is continuous on
\begin{align}\label{Scontset}
&\{ \text{$(\tilde \gamma, \mu)$ :  $\tilde \gamma$ is an equivalence class of simple curves and} \notag \\
&\qquad\qquad\qquad\text{the support of $\mu$ is $\tilde \gamma$}\}. 
\end{align}
\end{lemma}

\begin{proof}
To show that $T$ is continuous, observe that if $\supdist(\gamma_1, \gamma_2) = \delta$ then $\abmetric{\tilde{\gamma}_1}{\tilde{\gamma}_2} \leq \delta$, and if $\gamma_1(s) \in A$ for some Borel set $A$ then $\gamma_2(s) \in A^{\delta}$. Hence,
\begin{align*}
\nu_{\gamma_1}(A) = \int_{0}^{t_{\gamma_1}} \indicate{\gamma_1(s) \in A} \, ds 
&\leq \int_0^{t_{\gamma_1}} \indicate{\gamma_2(s) \in A^{\delta}} \, ds \\
&\leq |t_{\gamma_1} - t_{\gamma_2}| + \int_0^{t_{\gamma_2}} \indicate{\gamma_2(s) \in A^{\delta}} \,ds \\
&\leq \delta + \nu_{\gamma_2}(A^{\delta})
\end{align*}
which implies that $\levydist(\nu_{\gamma_1}, \nu_{\gamma_2}) \leq \delta$. This implies that under the metric~\eqref{productmetric},  $T$ is Lipshitz with Lipshitz constant no more than $2$.

To show that $S$ is continuous on the set~\eqref{Scontset}, suppose that $\gamma_1 , \gamma_2 \in \Omega$ are simple curves with $\abmetric{\gamma_1}{\gamma_2} \le \delta$. This implies that there exist representations $\eta_1 \in \tilde \gamma_1$ and $\eta_2 \in \tilde \gamma_2$ such that $t_{\eta_1} = t_{\eta_2}=1$ and
\begin{equation}\label{Sconteq0}
\sup_{0 \le s \le 1} |\eta_1(s) - \eta_2(s) | \le \delta
\end{equation}
which in turn implies that 
\begin{equation}\label{Sconteq1}
\eta_1[0,s] \subset (\eta_2[0,s])^\delta.
\end{equation}
Suppose further that $\mu_1, \mu_2 \in \mathcal{M}$ are supported on $\tilde \gamma_1$, $\tilde \gamma_2$, respectively, and satisfy $\levydist(\mu_1, \mu_2) \le \delta$ so that
\begin{equation}\label{Sconteq2}
\mu_1(\eta_1[0,s]) \le \mu_2( (\eta_1[0,s])^\delta) + \delta.
\end{equation}
Thus,~\eqref{Sconteq1} and~\eqref{Sconteq2} imply that
$\mu_1(\eta_1[0,s]) \le \mu_2( (\eta_2[0,s])^{2\delta}) + 2\delta$. 
Using the assumption that $\mu_2$ is supported on $\tilde \gamma_2$,  we find $\mu_2( (\eta_2[0,s])^{2\delta}) = \mu_2(\eta_2[0,s])$ and so
$\Theta_{\eta_1}(s) \le \Theta_{\eta_2}(s) + 2 \delta$
since $\Theta_{\eta_i}(s) = \mu_i(\eta_i[0,s])$ by definition.
Reversing the roles of $(\eta_1,\mu_1)$ and $(\eta_2,\mu_2)$ implies that
\begin{equation}\label{Sconteq3}
|\Theta_{\eta_2}(s) - \Theta_{\eta_1}(s)| \le 2 \delta.
\end{equation}
By construction, we have $\gamma_i(t) = \eta_i(\Theta_{\eta_i}^{-1}(t))$, independently of the choice of $\eta_i \in \tilde \gamma_i$, so that
$ t_{\gamma_i} = \mu_i(\eta_i[0,1]) =\mu_i(\D)$. Thus, we conclude that
\begin{equation}\label{Sconteq4}
|t_{\gamma_1} - t_{\gamma_2} | = | \mu_1(\D) - \mu_2(\D)| \le \delta
\end{equation}
since $\levydist(\mu_1, \mu_2) \le \delta$. The next step is to observe that since $t \mapsto \Theta_{\eta_1}(t)$ is continuous and non-decreasing, 
 \begin{align}\label{Sconteq5} \notag
&\sup_{0 \leq t \leq t_{\gamma_1} \vee \, t_{\gamma_2}} |\gamma_1(t) - \gamma_2(t)| \\ \notag
&\qquad=
 \sup_{0 \leq s \leq 1} |\gamma_1( \Theta_{\eta_1}(s)) - \gamma_2 (\Theta_{\eta_1}(s))| \\ \notag
&\qquad\le
 \sup_{0 \leq s \leq 1} |\gamma_1( \Theta_{\eta_1}(s)) - \gamma_2 (\Theta_{\eta_2}(s))| 
 +
 \sup_{0 \leq s \leq 1}|\gamma_2( \Theta_{\eta_2}(s)) - \gamma_2 (\Theta_{\eta_1}(s))|  \\ 
&\qquad=
 \sup_{0 \leq s \leq 1} |\eta_1(s) - \eta_2(s)| 
 +
 \sup_{0 \leq s \leq 1}|\gamma_2( \Theta_{\eta_2}(s)) - \gamma_2 (\Theta_{\eta_1}(s))| .
  \end{align}
We can use~\eqref{Sconteq0} to control the first term in the previous expression. In order to control the second term, we can use~\eqref{Sconteq3} along with the observation that $\gamma_2 \in \Omega$ is necessarily uniformly continuous. By combining~\eqref{Sconteq4} and~\eqref{Sconteq5} we conclude that $\supdist(\gamma_1,\gamma_2)$ can be made arbitrarily small so that $S$ is continuous on~\eqref{Scontset} as required.
\end{proof}

\subsection{The Lawler-Sheffield occupation measure for SLE}\label{LawSheffSect}

The last section shows that the problem of finding a natural time parameterization for SLE is equivalent to the problem of finding its natural occupation measure. For the loop-erased random walk and most other discrete models, the natural occupation measure counts the (scaled) number of steps of the discrete curve in each subset. Equivalently, but perhaps more simply stated, the natural occupation measure is just the standard Euclidean length measure on the curve. For SLE we would like to use a similar notion, but since the SLE curve is a fractal subset of the plane, it is not clear what the correct notion of length should be. 

 Recently, however, Lawler and Sheffield~\cite{LS09} have managed to construct what \textit{should} be the natural candidate for the length measure of SLE. Remarkably, their construction uses no geometric techniques and is purely probabilistic.  They have also proved that their measure has most of the natural properties that one would hope for in an SLE length measure, and that it is the unique such measure with these properties.   It is widely believed that their measure is the same as what one would get from any number of geometrical constructions. Even more recently, Lawler and Rezaei~\cite{LawRez2} were able to show that the Minkowski content of the chordal SLE exists and 
 agrees with Lawler and Sheffield's natural occupation measure.
 
 The uniqueness characterization, which we describe in more detail below, is extremely useful for our purposes. Most importantly it reduces our problem to showing that the length measure on loop-erased random walk converges to a limiting measure with the correct properties. Loosely speaking, the expected and desirable properties of the natural SLE occupation measure are the following.
\begin{enumerate}
\item \textbf{The occupation measure can be determined from the equivalence class $\tilde{\gamma}$.} This may seem slightly counterintuitive at first, but it is a manifestation of the idea that the natural occupation measure is a length measure on $\gamma$. That is, it should depend only on the geometry of the trace and \textit{not} on any particular initial time parameterization. For the loop-erased random walk this property holds trivially since the natural occupation measure is taken to be the length. The very construction of the Lawler-Sheffield natural occupation measure, which we will soon describe, guarantees that it also satisfies this property.

\item \textbf{The averaged occupation measure is absolutely continuous with respect to Lebesgue measure, and its density is the SLE Green's function.} By its very construction the SLE Green's function is the expected spatial density of the curve.  Recall its definition as
    \begin{align*}
    G(z) = \lim_{\epsilon \downarrow 0} \epsilon^{d - 2} \prob{\operatorname{d_{conf}}(z, \gamma) \leq \epsilon},
    \end{align*}
    where $d = 1 + \kappa/8$ is the dimension of the SLE curve and $\operatorname{d_{conf}}(z, \gamma)$ is one-half times the conformal radius of $z$ from $\gamma$. Ideally one would like to substitute the usual Euclidean distance for the conformal radius, but unfortunately it is not yet known that this limit exists in the radial case. For  chordal SLE,  the existence of the limit for conformal radius was proved in~\cite{Law09}, while  the existence of the limit  was proved in~\cite{LawRez2} for Euclidean distance. For  radial SLE and conformal radius the existence of the limit was proved in~\cite{AKL}.  In the case of chordal SLE, an exact formula for the Green's function is known for all values of $\kappa <8$. 
For radial SLE from a prescribed boundary point to a prescribed interior point, an exact formula  is known~\cite{AKL} only for $\kappa=4$. For other values of $\kappa$, including $\kappa=2$, the Green's function can be described in terms of an expectation with respect to radial SLE conditioned to go through a point~\cite{AKL}. For radial SLE started uniformly on $\bd \D$ and targeting the origin, the Green's function  is
 \begin{equation}\label{GFdisk}
 G_{\D}(z) = |z|^{d-2}.
 \end{equation}
 For other simply connected domains $D$ containing the origin, the Green's function is defined by the conformal covariance rule
  $$G_D(z) = |\phi'(z)|^{2-d} G_{\D}(\phi(z)) = \left| \frac{\phi'(z)}{\phi(z)}\right|^{2-d},$$
 where $\phi: D \to \D$ is a conformal transformation with $\phi(0)=0$.
  Equivalently, $G_D(z)$ is the Green's function for radial SLE in $D$ started with respect to harmonic measure on $\bd D$ and targeting the origin.

\item \textbf{The occupation measure has the domain Markov Property.} To readers already familiar with SLE this is not surprising, and the domain Markov property for the occupation measure is completely analogous to the one for the curve itself. It can be stated as follows: conditioned on some initial segment of the curve, the remaining measure has the law of the occupation measure corresponding to the remaining domain. The only issue left is how the SLE occupation measure is defined in simply connected domains other than the disk. Not surprisingly it satisfies a conformal \textit{covariance} rule:
    \begin{align*}
    \nu_{\psi(\gamma)}(\psi(A)) = \int_A |\psi'(z)|^{d} \nu_{\gamma}(dz),
    \end{align*}
    where $\psi$ is a conformal map from $\mathbb{D}$ onto a simply connected domain $D$. This implies an alternative form of the  domain Markov property: conditional on some initial segment of the curve $\gamma[0,t]$, the measure $\nu_t^*$ on $\D$ defined by
    \begin{align*}
    \nu_t^*(g_t(A)) = \int_A |g_t'(z)|^d \nu_{\gamma}(dz),
    \end{align*}
    for $A \subset \mathbb{D} \backslash \gamma[0,t]$, is independent of $\gamma[0,t]$ and has the same law as $\nu$.
By translating these properties through the mapping $S$ we get the following conformal covariance and domain Markov properties for the curve in the natural time parametrization. Let $\gamma$ be a naturally parametrized radial SLE in a simply connected domain $D$ from $z \in \bd D$ to $w \in D$. If $\phi: D \to D'$ is a conformal transformation, then
$$\gamma^*_t = \phi(\gamma_{\varsigma_t}), \;\;\; \text{where} \;\;\; t = \int_0^{\varsigma_t} |\phi'(\gamma_s)|^d ds,$$
is a naturally parametrized radial SLE in $D'$ from $\phi(z)$ to $\phi(w)$. Moreover, the law of $s \mapsto \gamma_{t+s}$ conditional on $\gamma[0,t]$ is a naturally parametrized SLE in the domain $D \setminus \gamma[0,t]$ from $\gamma(t)$ to $w$.
\end{enumerate}

For chordal SLE the existence of an occupation measure satisfying these properties has been proved in the series of papers~\cite{LawRez,LawRez2,LS09,LawZhou}. No work has been done in the radial case, but it is widely expected that the analogous results still hold. We record the following conjecture for the occupation measure of radial SLE($2$). It is expected that it is true for all $\kappa \leq 4$, with minor modifications for $4 < \kappa < 8$.

\begin{conjecture}\label{conj:existence}
There exists a probability measure on the space $\tilde{\Omega} \times \mathcal{M}$ such that for a pair $(\tilde{\gamma}, \mu)$
\begin{enumerate}
\item $\tilde{\gamma}$ is an equivalence class of SLE(2) curves on $\mathbb{D}$,
\item $\mu$ is measurable with respect to $\tilde{\gamma}$,
\item for all $\gamma \in \tilde{\gamma}$, $\mu(\cdot \cap \gamma[0,t])$ is measurable with respect to $\tilde{\F}_t(\gamma) = \widetilde{\gamma[0,t]}$,
\item $\expect{\mu(dz)} = G(z) \, dz$ where $G(z) = G_{\D}(z)$ as in~\eqref{GFdisk}, which is understood to mean
$$\expect{\mu(A)} = \int_A G(z) dz $$
for every Borel $A \subset \D$, and
\item the domain Markov property holds.
\end{enumerate}
\end{conjecture}

As mentioned, proving existence in the chordal case was a very challenging problem~\cite{LawRez,LawRez2,LS09,LawZhou}, and although the corresponding radial result is expected to hold, it will still be a technical challenge. For us, however, we are more interested in the uniqueness properties. It turns out that the measure of the last conjecture is the \textbf{unique} measure satisfying the five conditions above, and this characterization of the measure is extremely important for the remainder of our paper.

\begin{theorem}\label{thm:uniqueness}
If the probability measure of Conjecture~\ref{conj:existence} exists then it is unique.
\end{theorem}

\begin{proof}[Proof of uniqueness assuming existence]
Let $\mu$ be a positive random measure on $\D$, and assume that $\expect{\mu(\D)} < \infty$. The latter will hold for any measure satisfying condition 4. Then $\E[ \mu(A) | \tilde{\F}_t]$ is a martingale for each Borel $A \subset \D$ (where $\tilde{\F}_t = \tilde{\F}_t(\gamma)$), and assuming condition 2 it follows that $\E[ \mu(A) | \tilde{\F}_t] \to \mu(A)$ almost surely as $t \to \infty$. Now assuming condition 3 we may write
\begin{align}\label{eqn:doob_meyer_decomp}
\E[ \mu(A) | \tilde{\F}_t ] = \E[ \mu(A \backslash \gamma[0,t]) | \tilde{\F}_t ] + \mu(A \cap \gamma[0,t]).
\end{align}
The left side of~\eqref{eqn:doob_meyer_decomp} is a martingale and the second term on the right is clearly non-decreasing in $t$, hence the remaining term must be a supermartingale. Equation~\eqref{eqn:doob_meyer_decomp} is the Doob-Meyer decomposition for the supermartingale term (see~\cite{DM82} for a detailed treatment), and it is well known that this decomposition is unique.

In the remaining we will show that conditions 4 and 5 uniquely determine what the supermartingale term must be, and then the uniqueness of the Doob-Meyer decomposition determines the process $\E[\mu(A) | \tilde{\F}_t]$. Taking $t \to \infty$ uniquely gives $\mu(A)$ (again by condition 2), and since this procedure can be repeated for a countable, measure-determining collection of Borel sets $A$ we have uniquely determined $\mu$.

To derive the supermartingale term first define the measure $\mu_t^*$ on $\D$ by
\begin{align}\label{eqn:DM_def}
\mu_t^*(g_t(E)) = \int_E |g_t'(z)|^d \, \mu(dz)
\end{align}
for all Borel $E \subset \D \backslash \gamma[0,t]$. By the domain Markov property of condition 5, $\mu_t^*$ has the same law as $\mu$ but is independent of $\gamma[0,t]$. Combining this with condition 4 implies that
\begin{align*}
\E[\mu_t^*(g_t(E)) | \tilde{\F}_t] = \int_{g_t(E)} \!\!\!\!\!\! G(w) \, dw = \int_E |g_t'(z)|^2 G(g_t(z)) \, dz,
\end{align*}
the last equality following by a standard change of variables. But by the definition of $\mu_t^*$ in~\eqref{eqn:DM_def} we also have that
\begin{align*}
\E[\mu_t^*(g_t(E)) | \tilde{\F}_t] = \int_E |g_t'(z)|^d \E[\mu(dz) | \tilde{\F}_t].
\end{align*}
Equating the last two expressions and observing that they hold for all $E$ forces that
\begin{align}\label{eqn:cond_term}
\E[\mu(dz) | \tilde{\F}_t] = |g_t'(z)|^{2-d} G(g_t(z)) \, dz, \;\;\; z \in \D \setminus \gamma[0,t].
\end{align}
Therefore the unique choice for the desired supermartingale term is
\begin{align}\label{eqn:super_mart_term}
\E \left[ \mu(A \backslash \gamma[0,t]) \right] = \int_{\D} |g_t'(z)|^{2-d} G(g_t(z)) \indicate{z \in A \backslash \gamma[0,t]} \, dz.
\end{align}
This completes the proof of uniqueness.
\end{proof}

To prove the existence part of Conjecture~\ref{conj:existence} it must be shown that~\eqref{eqn:super_mart_term} admits a Doob-Meyer decomposition as a martingale minus a non-increasing process that is not identically zero. It is not immediate that this is true. For example, the term $|g_t'(z)|^{2-d} G(g_t(z))$ is a positive \textit{local martingale} (and hence a supermartingale) that blows up if the curve reaches $z$, or stays bounded and goes to zero when the curve reaches the origin. Since we are dealing with SLE($2$) the curve never reaches a fixed point $z$ and hence the local martingale evolves as a true martingale; its increasing part is therefore zero. Equation~\eqref{eqn:super_mart_term} has a non-trivial increasing part only because the domain of integration decreases as $t$ increases. This is the intuition behind the existence part of Conjecture~\ref{conj:existence} but it requires several difficult estimates to rigorously prove it.

Finally, given the occupation measure, the SLE curve in the natural time parameterization is constructed using the map $S$ defined in the last section. This is equivalent to the following: assuming the existence of the decomposition
\begin{align*}
\int_{\D} |g_t'(z)|^{2-d} G(g_t(z)) \indicate{z \in \D \backslash \gamma[0,t]} \, dz = M_t - A_t
\end{align*}
where $M_t$ is a martingale and $A_t$ is an increasing process with $A_0 = 0$, one expects that $A_t$ is continuous and strictly increasing (this should be proved along with the existence in Conjecture~\ref{conj:existence}). Hence it has a continuous inverse $\Theta_{\gamma}(t)$, i.e., $A_{\Theta_{\gamma}(t)} = t$, and the SLE curve in the natural time parameterization is defined by
\begin{align*}
\gamma^*(t) = \gamma(\Theta_{\gamma}(t)).
\end{align*}
Note that the decomposition above is for a particular representative of $\gamma \in \tilde{\gamma}$ (for example, the capacity parameterization) but that $\gamma^*$ is ultimately independent of the choice of representative. Indeed, if $\eta \in \tilde{\gamma}$ is any other representation, then it is a reparameterization of $\gamma$, i.e., $\eta = \gamma \circ \phi$ for some increasing homeomorphism $\phi$, and a Loewner chain for $\eta$ is a time change of the Loewner chain for $\gamma$, i.e., $g_t^{\eta} = g_{\phi(t)}$. Hence the supermartingale term above undergoes a simple time change, and it is standard that time changing the supermartingale term only time changes the corresponding Doob-Meyer decomposition (since time-changed martingales are still martingales and the decomposition is unique). Therefore the martingale term and the increasing part are reparameterized in the same way as the supermartingale term, i.e.,
\begin{align*}
\int_{\D} |g_{\phi(t)}'(z)|^{2-d} G(g_{\phi(t)}(z)) \indicate{z \in \D \backslash \eta[0,t]} \, dz = M_{\phi(t)} - A_{\phi(t)}.
\end{align*}
The inverse to the increasing process for $\eta$ is therefore just a time change of the inverse to the increasing process corresponding to $\gamma$, in other words $\Theta_{\eta} = \phi^{-1} \circ \Theta_{\gamma}$. Therefore
\begin{align*}
\eta \circ \Theta_{\eta} = \gamma \circ \phi \circ \phi^{-1} \circ \Theta_{\gamma} = \gamma \circ \Theta_{\gamma},
\end{align*}
hence $\gamma^*$ is indeed well-defined.

\section{Convergence in the natural time parameterization}\label{SectConvStrat}

We now describe our strategy for showing weak convergence of $Y_n$ to SLE($2$) with the natural time parameterization. The topology is the one induced by the distance function~\eqref{strongmetric}. We emphasize that we do not actually prove the weak convergence of $Y_n$ in this topology, but we do determine which results are needed so that, if proved, the weak convergence of $Y_n$ would hold. 

Recall that $Y_n(t) = n^{-1}X(\sigma_n(t))$ where  $X(t)$, $0\le t \le M_n$, denotes the time reversal of a  loop-erased random walk on $\Z^2$ started at the origin and stopped at $M_n$, the time the loop-erased random walk reaches the circle of radius $n$. Here $\sigma_n(t)$ is a continuous, strictly increasing reparametrization of time depending on the lattice spacing; this is what we call the speed function. The usual choice is $\sigma_n(t) = c_n t$ for some $c_n >0$, so that the loop-erased random walk moves at constant speed for its entire lifetime; typically, $c_n = n^{5/4}$ or $c_n = \expect{M_n}$. Note, however, that our strategy  is not restricted to this class of speed functions. All that we require is that the induced occupation measure is measurable with respect to the trace of the loop-erased random walk.  For  the class $\sigma_n(t) = c_n t$, this is trivially satisfied since the induced occupation measure is a simple rescaling of the arclength. 
Our strategy is very general  and should  also be applicable for showing convergence of discrete lattice curves from other models to SLE in the natural time parameterization.

The main idea behind our strategy is summarized in Figure~\ref{convergence_diagram}. Our goal is to prove the convergence on the bottom part of the diagram, but we do so by showing the convergence along the top. Using the map $T$ we send $Y_n$ to the pair $(\tilde{Y}_n, \nu_{Y_n})$, and then our main goal is to show weak convergence of this pair to the SLE($2$) equivalence class and the Lawler-Sheffield occupation measure. That convergence along the top part of the diagram implies convergence along the bottom is a consequence of the Continuous Mapping Theorem. The only technical requirement is that the pair $(\tilde{\gamma}, \mu)$ consisting of the SLE($2$) equivalence class and its occupation measure satisfy the conditions of Lemma~\ref{lemma:TScontinuity}, but this implied by Conjecture~\ref{conj:existence}.

Showing the convergence of $(\tilde{Y}_n, \nu_{Y_n})$ is accomplished in several steps that are roughly equivalent to the ones used for showing convergence of $\tilde{Y}_n$. We begin by showing tightness of the pair which, by Prokhorov's theorem, implies the existence of subsequential limits. With those in hand the goal becomes proving that all subsequential limits are the same, and this is done by showing that all possible limits satisfy conditions 1 through 5 of Conjecture~\ref{conj:existence}. The uniqueness statement of Theorem~\ref{thm:uniqueness} then implies that all the limits coincide.

In the rest of this section we outline the ideas behind showing the tightness and that all subsequential limits coincide.

\begin{figure}
\begin{center}
\includegraphics[width=5.5in]{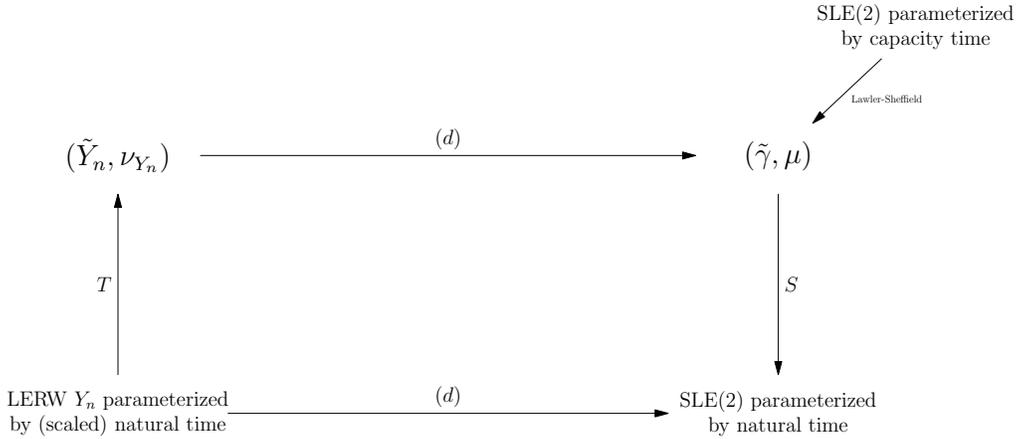}
\caption{Strategy for convergence in the natural time parameterization \label{convergence_diagram}}
\end{center}
\end{figure}

\subsection{Tightness}

The first step is to prove tightness of the pair $(\tilde{Y_n}, \nu_{Y_n})$. Since we are working under the product topology it is sufficient to prove tightness of each individual coordinate variable. 
Tightness of $\tilde{Y_n}$ follows from Lawler-Schramm-Werner~\cite{LSW04}, hence what remains is tightness of the occupation measures. Since we are working on a bounded domain it is sufficient that the total mass of the measures is also tight. More precisely, for $K > 0$ the subsets $\mathcal{M}_K = \{ \mu \in \mathcal{M} : \mu(\D) \leq K \}$ are compact in $\mathcal{M}$, hence it is sufficient to show that $\nu_{Y_n}(\D)$ is a tight random variable. It follows by definition that $\nu_{Y_n}$ is the lifetime of the curve $Y_n$, which is in turn determined by the choice of the speed function $\sigma_n$. 
For $\sigma_n(t) = \E[M_n] t$, Barlow and Masson~\cite{BM09} show that the lifetime $M_n/\E[M_n]$ is tight. 
For $\sigma_n(t) = n^{5/4} t$, there are no known tightness results on the lifetime $n^{-5/4}M_n$, although work in progress\footnote{Lawler, G.F., Johansson Viklund, F.: In progress, 2013.} might change this. 

We point out that tightness of the lifetime of $Y_n$ is also a necessary condition for tightness of $\nu_{Y_n}$. The map $\mu \mapsto \mu(\D)$ from $\mathcal{M} \to \R$ is continuous in the L\'{e}vy-Prokhorov metric, hence tightness of $\nu_{Y_n}$ implies tightness of $\nu_{Y_n}(\D)$. Moreover, since we are working under the product topology the tightness of $\nu_{Y_n}(\D)$ is also necessary for the tightness of the pair $(\tilde{Y}_n, \nu_{Y_n})$.

In the rest of this section we assume that the pair $(\tilde{Y}_n, \nu_{Y_n})$ is tight, and we use $(\tilde Y, \nu_Y)$ to denote any subsequential limit. We devote the rest of this section to outlining a strategy for showing subsequential limits satisfy conditions 1 through 5 of Conjecture~\ref{conj:existence}.

\subsection{Measurability conditions}

To establish condition 2 of Conjecture~\ref{conj:existence} one must show that for any subsequential limit $(\tilde Y, \nu_Y)$, the first coordinate is measurable with respect to the second. For most reasonable choices of the discrete parametrization, this will hold for $(\tilde Y_n, \nu_{Y_n})$, in particular, if $\nu_{Y_n}$ is some variant of arclength measure on the curve.

However, even if $\nu_{Y_n}$ is measurable with respect to $\tilde Y_n$ for each $n$, it does not follow that this persists in the limit. Measurability properties of this type are not necessarily preserved under weak convergence. For example, consider the measure on $[0,1] \times [-1,1]$ induced by $(U,\sin 2\pi n U)$ where $U$ is uniformly distributed on $[0,1]$. Clearly, the second coordinate is measurable with respect to the first for each $n$, but the same does not hold for the weak limit of the induced measure.

There are two standard techniques for establishing this type of measurability. The first, which is used in~\cite{JDub}, is to show that conditionally on $\tilde Y$, the $\nu_Y$ variable is independent of itself and hence a constant. The second approach, which is the one used in~\cite{GPS}, is to show that on the discrete level $\nu_{Y_n}(A)$ can be well-approximated by another random variable $Z_n(A)$ that is measurable with respect to $\tilde Y_n$, and that converges to a quantity that is also measurable with respect to $\tilde Y$. In~\cite{GPS} ``well-approximated'' means that $\nu_{Y_n}(A) - Z_n(A)$ goes to 0 in the $L^2$ sense.

For the present paper we believe the second approach is likely to be simpler to implement. Although it will require many rigorous estimates, we expect one can choose $Z_n(A)$ to be the increasing part of the Doob-Meyer decomposition of
$$\int_A | (g_t^{(n)})'(z)|^d G(g_t^{(n)}(z)) dz,$$
where $g_t^{(n)}$ is the Loewner map associated with the discrete curve $Y_n$. By convergence of $\tilde Y_n$ to $\tilde \gamma$ we expect that this quantity will be close to the analogous quantity for $\tilde \gamma$, and that measurability will follow by similar considerations.

\subsection{Conformal invariance and the domain Markov property}

Thus far, we have only discussed convergence of loop-erased random walk to SLE(2) in $\D$, but to carry out the full strategy one must tackle substantially more. As in the original paper by Lawler, Schramm, and Werner~\cite{LSW04}, to prove convergence in a single domain it is required to prove convergence in all domains. This is due to the nature of the strategy: one proves that the laws of the random objects have subsequential limits, and then shows that all subsequential limits share certain characterizing properties that identify them as being the same.  For convergence modulo reparametrization one of the characterizing properties is conformal invariance of the curves. For convergence with the time parametrization taken into account this is replaced by conformal covariance of the random occupation measures.

The conformal covariance rule for the random measures is  discussed in Section~\ref{LawSheffSect}. A first step to establishing the conformal covariance is to show that for any subsequential limit on a given domain the expectation is the SLE Green's function on that domain. If this can be proved then it can be applied to show that conditional on $\tilde\F_t(Y)$ where $\tilde Y$ is the subsequential limit, the expected density of the remaining measure is the SLE Green's function in the remaining domain. This would imply a proof of~\eqref{eqn:cond_term} which, by the argument outlined in the proof of Theorem~\ref{thm:uniqueness} and the uniqueness of the Doob-Meyer decomposition, implies the domain Markov property for the occupation measures.

Hence, we outline a strategy for showing that on any simply connected domain, the expectation of any subsequential limit of the random occupation measures is the SLE Green's function on that domain. Let $D \subseteq \C$ be a simply connected domain with $0 \in D$, and let $\phi: D \to \D$ be the unique conformal transformation from $D$ to $\D$ with $\phi(0)=0$ and $\phi(0)>0$. Let $D^n$ be the $n^{-1}\Z^2$ grid domain approximation of $D$; that is, the connected component containing the origin of the complement of the closed faces in $n^{-1}\Z^2$ intersecting $\bd D$. Let $Y^n_D$ be the time reversal of loop-erased random walk on $n^{-1}\Z^2$ started at 0 and stopped when hitting $\bd D^n$. The time parametrization of $Y^n_D$ is chosen so that
$$|\bd_t Y_D^n(t)| = c_n$$
for all $t>0$ where $c_n$ is a constant that depends only on the lattice spacing.

Write $\nu_D^n := \nu_{Y^n_D}$ for the induced occupation measure on the edges of $n^{-1}\Z^2$ induced by the curve $Y^n_D$; this is a scaled version of the arclength measure on the trace of $Y^n_D$. 

\begin{conjecture}
The measures $\expect{\nu^n_D}$ converge weakly to the measure on $D$ whose density with respect to Lebesgue measure is
$$G_D(z) := \left|\frac{\phi'(z)}{\phi(z)}\right|^{3/4}.$$
\end{conjecture}

In what we have written above, we have chosen $c_n$ independently of the domain of consideration. The standard choices for $c_n$ are $c_n = n^{5/4}$ or $c_n = \expect{M_n}$.
This is analogous to the situation for random walks where the speed of the discrete curve is determined entirely by the lattice spacing, and not by the domain.
Letting $M_D^n$ be the number of steps in $Y_D^n$, the lifetime of the discrete curve is  either $n^{-5/4}M_D^n$ or $M_D^n / \expect{M_n}$. Again tightness results  for the first quantity are currently not known, but tightness for the second is proved in Theorem~1.1 of Barlow and Masson~\cite{BM09}.

\subsection{Limits of the expected measure}

In this section, we outline the steps for showing any subsequential limit satisfies condition~4 of Conjecture~\ref{conj:existence}; that is, 
$\expect{\nu_Y(dz)} = G(z) \, dz$ where $G(z) = G_\D(z)$ as given by~\eqref{GFdisk}. 
We discuss this for the special case of loop-erased random walk in $\D$. 
For any Borel $A \subset \D$, we have
$$\nu_{Y_n}(A) = \int_0^\infty \indicate{Y_n(s) \in A, \; s < t_{Y_n}} ds$$
so that by Fubini's theorem,
\begin{equation}\label{discretetimeint}
\expect{\nu_{Y_n}(A)} = \int_0^\infty \prob{Y_n(s) \in A, \; s < t_{Y_n}} ds.
\end{equation}
Showing that $\nu_Y$ satisfies condition 4 is equivalent to showing that the latter integral converges to
$$\int_A G(z) dz.$$
Note that~\eqref{discretetimeint} is an integral over time which we first need to convert into an integral over space. 
Consider the case when $\sigma_n(t) = c_n t$. Then the occupation measure for $\tilde Y_n$ is uniform along the edges that $\tilde Y_n$ traverses with density $c_n^{-1}$ with respect to arclength. Hence,
\begin{equation}\label{discretetimeint2}
\nu_{Y_n}(A) = \frac{1}{c_n} \sum_{e \in A_n} \indicate{z_e \in \tilde Y_n},
\end{equation}
where $A_n = A \cap n^{-1}\Z^2$ so that the sum is over all (undirected) edges $e$ of $A_n$, and $z_e$ is the midpoint of the edge $e$. 
Observe that $\nu_{Y_n}(\D) = c_n^{-1}M_n$. Taking expectations in~\eqref{discretetimeint2} yields
$$\expect{\nu_{Y_n}(A)} = \frac{1}{c_n} \sum_{e \in A_n} \prob{z_e \in \tilde Y_n}= \frac{1}{2n^2} \sum_{e \in A_n} \frac{2n^2}{c_n} \prob{z_e \in \tilde Y_n}.$$
The extra factor of $2n^2$ in the  rightmost equation above makes the summation look like a Riemann sum approximation to an integral.  The spatial function $z \mapsto 2n^2c_n^{-1} \prob{z \in \tilde Y_n}$, defined on edges of $n^{-1}\Z^2$, is being sampled at the midpoint $z_e$ of each edge and represents the value of the function on the unique regular diamond  (i.e., square rotated by 45$^\circ$) of side length $2^{-1/2}n^{-1}$  centred at the midpoint.

Since $G(z)$ is Riemann integrable so that
$$\frac{1}{2n^2} \sum_{e \in A_n} G(z_e) \to \int_A G(z) dz$$
as $n \to \infty$, 
it is sufficient to show that
$$ \sum_{e \in A_n} \left[ \frac{2n^2}{c_n} \prob{z_e \in \tilde Y_n} - G(z_e) \right]  = o(n^2)$$

Carrying out this estimate appears to be genuinely difficult. It is a hard problem to describe asymptotics for the probability that loop-erased walk passes through a particular edge, and even harder to show that the limit is the SLE Green's function. There appears to be work in progress\footnote{Lawler, G.F., Johansson Viklund, F.: In progress, 2013.} in this direction, although it is not yet clear how sharp the aymptotics will be.

In the special case that $A =  \ball{z}{\epsilon} \subset \D$, we believe the simple geometry of $A$ can be helpful. Write
$$\expect{\nu_{Y_n}(\ball{z}{\epsilon})} = \condexpect{\nu_{Y_n}(\ball{z}{\epsilon})}{\tilde Y_n \cap \ball{z}{\epsilon} = \emptyset} \prob{\tilde Y_n \cap \ball{z}{\epsilon} = \emptyset} .$$
By~\eqref{LSWconseq}, the second term on the right above converges to $\prob{\tilde \gamma \cap \ball{z}{\epsilon} = \emptyset}$. For the first term, we roughly expect that the loop-erased walk goes through $\ball{z}{\epsilon}$ as if it were a loop-erased walk in that domain, i.e., it should not be influenced too much by its future or past. This leads to the conjecture that
$$\condexpect{\nu_{Y_n}(\ball{z}{\epsilon})}{Y_n \cap \ball{z}{\epsilon} \neq \emptyset} = \frac{\expect{M_{\epsilon n}}}{\expect{M_n}} + o(1).$$
In the next section we will discuss this conjecture in more detail; see Conjecture~\ref{LERWconj2}. In Theorem~\ref{LERWupperboundThm} we have established a rigorous result for loop-erased random walk that provides evidence for this conjecture.

To complete the convergence to the Green's function via this strategy, we also expect that
$$\lim_{n \to \infty} \frac{\expect{M_{\epsilon n}}}{\expect{M_n}} = \epsilon^{5/4} [1 + o(1)] $$
as $\epsilon \to 0$. Combining with the previous estimates, this will show that
\begin{align*}
\lim_{n \to \infty} \expect{\nu_{Y_n}(\ball{z}{\epsilon})} &= \epsilon^{5/4} \prob{\tilde \gamma \cap \ball{z}{\epsilon} = \emptyset} [1+o(1)]\\
&=  e^{2}G(z) [1+o(1)].
\end{align*}

\section{Results and conjectures for loop-erased random walk}\label{lerwsect}

In this section we establish certain technical estimates for loop-erased random walk.  Many of these estimates require results due to Barlow and Masson~\cite{BM09} and the reader is referred to that paper for further details. Unfortunately, we are not able to establish the loop-erased random walk estimates in as strong a form as is needed to prove convergence to the Lawler-Sheffield natural measure. In particular, while we conjecture that certain estimates hold up to constants, we are only able to rigorously prove upper bounds. We believe that the lower bounds also  hold and we outline a possible approach to their proofs. 

Throughout this section, the letters $c$ and $C$ will be used to denote positive constants that do not depend on any variable, but may, however, change from line-to-line.  If $f(n)$ and $g(n)$ are two positive functions, then we write $f(n) \asymp g(n)$ if there exists a $C<\infty$ such that
$$C^{-1} g(n) \le f(n) \le C g(n)$$
for all $n$.
Let $B_n = \{ x \in \Z^2 : |x| \le n\}$ denote the (discrete) ball of radius $n$ centred at the origin. By $\Es(n)$ we mean the probability that a random walk from the origin to $\bd B_n$ and the loop-erasure of an independent random walk from the origin to $\bd B_n$ do not intersect. By $\Es(m,n)$ we mean the probability that a random walk from the origin to $\bd B_n$ and the terminal part of an independent loop-erased random walk from $m$ to $n$ do not intersect.

\begin{conjecture}\label{LERWconj1}
If $z \in \mathbb{D}$ and $\epsilon > 0$ is sufficiently small, then
\begin{equation*}
\condexpect{\nu_{Y_n}(\ball{z}{\epsilon})}{Y_n \cap \ball{z}{\epsilon} \neq \emptyset}\asymp \frac{\expect{M_{\epsilon n}}}{\expect{M_n}}.
\end{equation*}
\end{conjecture}

Concerning the upper bound in Conjecture~\ref{LERWconj1} we are able to rigorously establish that if $z \in \mathbb{D}$ and $\epsilon > 0$ is sufficiently small, then
\begin{equation}\label{LERWupperboundstatement}
\condexpect{\nu_{Y_n}(\ball{z}{\epsilon})}{Y_n \cap \ball{z}{\epsilon} \neq \emptyset} \leq C \log(1/\epsilon) \frac{\expect{M_{\epsilon n}}}{\expect{M_n}}.
\end{equation}

Equation~\eqref{LERWupperboundstatement} is an immediate consequence of the following theorem.  We have chosen to express the theorem as a statement solely about loop-erased random walk, and not as a statement about the occupation measure for loop-erased random walk.
Indeed, suppose that $X$ is the reversal of a loop-erased random walk on $\Z^2$ started from 0 stopped when exiting the ball of radius $n$, and let $X^n = X/n$ so that $X^n$ is the reversal of a loop-erased random walk on $n^{-1}\Z^2$ started from the origin stopped when exiting the ball of radius $1$.
If $Y_n(t) = X^n(\sigma_n(t))$ where $\sigma_n(t) = \expect{M_n}t$, then
\begin{align*}
&\condexpect{\nu_{Y_n}(\ball{z}{\epsilon})}{Y_n \cap \ball{z}{\epsilon} \neq \emptyset}\\
&\qquad=
\frac{
\condexpect{\text{\# of steps of $X^n$ in $\ball{z}{\epsilon}$}}{X^n \cap \ball{z}{\epsilon} \neq \emptyset}
}{\expect{M_n}}.
\end{align*}

\begin{theorem}\label{LERWupperboundThm}
If $z \in \mathbb{D}$ and $\epsilon > 0$ is sufficiently small, then
\begin{equation} \label{may24eq0}
\condexpect{\text{\# of steps of $X^n$ in $\ball{z}{\epsilon}$}}{X^n \cap \ball{z}{\epsilon} \neq \emptyset} \leq C \log(1/\epsilon) \expect{M_{\epsilon n}}.
\end{equation}
\end{theorem}

\begin{proof}
Let $\tau$ be the first time that $X^n$ hits $\ball{z}{\epsilon}$, let $\alpha = X^n[0,\tau]$, and let $x = X^n(\tau)$. Then by the domain Markov property for loop-erased random walk, conditioned on $\alpha$, the rest of $X^n$ is obtained by running a random walk $W$ (on $n^{-1}\Z^2$) started at $x$ conditioned to leave $\mathbb{D}$ before hitting $\alpha$.

Now fix a point $w \in \ball{z}{\epsilon/2}$. Let $\widetilde{Z}$ be $W$ started at $x$ conditioned to hit $w$ before leaving $\mathbb{D}$,  and let $\rho$ be the last visit to $w$ by $\widetilde{Z}$. It then follows from Proposition~5.2 of~\cite{BM09} with $k=1$ that
\begin{equation}\label{may24eq1}
\condprob{w \in X^n}{X^n[0,\tau] = \alpha} = G^W(x,w)\prob{\Lo(\widetilde{Z}[0,\rho]) \cap W^w[1, \sigma_{\mathbb{D}}] = \emptyset}
\end{equation}
where $G^W(x,w)$ is the simple random walk Green's function for $W$ in $\mathbb{D}$, $\Lo$ is the loop-erasure operator, $W^w$ is an independent copy of $W$ started at $w$, and $\sigma_{\mathbb{D}}$ is the first exit time of $\mathbb{D}$ by $W^w$.

Recall that one obtains the same distribution by loop-erasing forwards and backwards, and therefore if we let $Z$ be a random walk started at $w$ conditioned to leave the domain $\mathbb{D} \setminus \alpha$ at $x$, then
$$\prob{\Lo(\widetilde{Z}[0,\rho]) \cap W^w[1, \sigma_{\mathbb{D}}] = \emptyset} = \prob{\Lo(Z[0,\xi_x]) \cap W^w[1, \sigma_{\mathbb{D}}] = \emptyset}$$
where $\xi_x$ is the first hitting time of $x$ by $Z$. Hence, we can rewrite~\eqref{may24eq1} as
\begin{equation}\label{may24eq2}
\condprob{w \in X^n}{X^n[0,\tau] = \alpha} = G^W(x,w)\prob{\Lo(Z[0,\xi_x]) \cap W^w[1, \sigma_{\mathbb{D}}] = \emptyset}.
\end{equation}
Our goal now is to estimate the two terms on the right side of~\eqref{may24eq2}, namely
$G^W(x,w)$ and $\prob{\Lo(Z[0,\xi_x]) \cap W^w[1, \sigma_{\mathbb{D}}] = \emptyset}$.

We begin by observing that the estimate for $G^W(x,w)$ is similar to Lemma~4.6 of~\cite{BM09}. That is, if $\sigma_D$ is the first exit time of a set $D$ and $\xi_D$ is the first hitting time of $D$, then by basic potential theory (see Lemma~2.1 of~\cite{BM09}) it follows that
\begin{align}\label{may24eq3}
G^W(x,w) &= \frac{\probb{w}{\sigma_{\D} < \xi_\alpha}}{\probb{x}{\sigma_{\D} < \xi_\alpha}} G(x,w; \mathbb{D} \setminus \alpha) \notag\\
&=  \frac{\probb{w}{\sigma_{\D} < \xi_\alpha} }{\probb{x}{\sigma_{\D} < \xi_\alpha}}  \probb{x}{\xi_w < \xi_\alpha \wedge \sigma_{\mathbb{D}}} G(w,w; \mathbb{D} \setminus \alpha)
\end{align}
where $G(\cdot , \cdot ; \cdot )$ is the Green's function for simple random walk.
In order to estimate the right side of~\eqref{may24eq3}, we note that
\begin{equation}\label{may24eq2b}
\log(\epsilon n) \asymp G(w,w; \ball{w}{\epsilon/2})  \leq G(w,w; \mathbb{D} \setminus \alpha) \leq G(w,w; \mathbb{D}) \asymp \log n.
\end{equation}
Next, we observe that
\begin{align*}
&\probb{x}{\xi_w < \xi_\alpha \wedge \sigma_{\mathbb{D}}}\\
&\quad =\!\! \sum_{y \in \partial \ball{w}{\epsilon/8}} \!\probb{y}{\xi_w < \xi_\alpha \wedge \sigma_{\D}} \probb{x}{S(\xi_{\ball{w}{\epsilon/8}}) = y; \, \xi_{\ball{w}{\epsilon/8}}  < \xi_\alpha \wedge \sigma_{\D}}.
\end{align*}
However, if $y \in \partial \ball{w}{\epsilon/8}$, then
\begin{equation*}
\probb{y}{\xi_w < \xi_\alpha \wedge \sigma_{\D}} \leq \probb{y}{\xi_w < \sigma_{\D}} \leq C \frac{\log(1/\epsilon)}{\log n}
\end{equation*}
and so
\begin{equation} \label{s4eq2}
 \frac{\probb{x}{\xi_w < \xi_\alpha \wedge \sigma_{\mathbb{D}}}}{\probb{x}{\xi_{\ball{w}{\epsilon/8}} < \xi_\alpha \wedge \sigma_{\D}}} \leq C \frac{\log(1/\epsilon)}{\log n}.
\end{equation}
It follows by the discrete Harnack principle that
\begin{equation} \label{s4eq3}
\probb{x}{\sigma_{\D} < \xi_\alpha} \geq c \probb{w}{\sigma_\D < \xi_\alpha} \probb{x}{\xi_{\ball{w}{\epsilon/8}} < \xi_\alpha \wedge \sigma_\D }.
\end{equation}
Hence, combining~\eqref{s4eq2} and~\eqref{s4eq3} gives
\begin{equation}\label{Mar21eq1}
\frac{\probb{w}{\sigma_{\D} < \xi_\alpha} \probb{x}{\xi_w < \xi_\alpha \wedge \sigma_{\mathbb{D}}} }{\probb{x}{\sigma_{\D} < \xi_\alpha}} \leq C \frac{\log(1/\epsilon)}{\log n},
\end{equation}
and so combining this with~\eqref{may24eq2b} gives an upper bound for~\eqref{may24eq3}, namely
\begin{equation*}
G^W(x,w) \leq C \log(1/\epsilon)
\end{equation*}
and completes the first part of the estimate for the right side of~\eqref{may24eq2}.

We will now finish estimating the right side of~\eqref{may24eq2} by showing that
\begin{equation}
\label{s4eq1}
\prob{\Lo(Z[0,\xi_x]) \cap W^w[1, \sigma_{\mathbb{D}}] = \emptyset} \le C\Es(\epsilon n).
\end{equation}
 Let $\beta^1$ be $\Lo(Z[0,\xi_x])$ from $w$ up to its first exit of $\ball{w}{\epsilon/8}$ and $\beta^2$ be $W^w[1,\sigma_{\mathbb{D}}]$ up to its first exit of $\ball{w}{\epsilon/8}$. Then,
\begin{equation*}
\prob{\Lo(Z[0,\xi_x]) \cap W^w[1, \sigma_{\mathbb{D}}] = \emptyset} \leq \prob{\beta^1 \cap \beta^2 = \emptyset}.
\end{equation*}
By the Harnack principle, $\beta^2$ has the same distribution up to constants as simple random walk started at $w$ stopped at its first exit of $\ball{w}{\epsilon/8}$. Furthermore, by Corollary~3.4 of~\cite{BM09}, $\beta^1$ has the same distribution up to constants as infinite loop-erased random walk started at $w$ stopped at its first exit of $\ball{w}{\epsilon/8}$. Finally, it follows from Theorem~3.9 and Lemma~3.10 of~\cite{BM09} that
\begin{equation*}
\prob{\beta^1 \cap \beta^2 = \emptyset} \asymp \Es(\epsilon n / 8) \asymp \Es(\epsilon n)
\end{equation*}
which establishes~\eqref{s4eq1} and completes the estimate for the right side of~\eqref{may24eq2}.

The final step to completing the proof is to establish~\eqref{may24eq0}. This follows since
\begin{align*}
&\condexpect{\text{\# of steps of $X^n$ in $\ball{z}{\epsilon}$}}{X^n \cap \ball{z}{\epsilon} \neq \emptyset}\\
&\qquad\qquad\leq \sum_{w \in \ball{z}{\epsilon}} C \, \condprob{w \in X^n}{X^n \cap \ball{z}{\epsilon/2} \neq \emptyset} \\
&\qquad\qquad\leq C \sum_{w \in\ball{z}{\epsilon}} \log(1/\epsilon) \Es(\epsilon n) \\
&\qquad\qquad \leq C \log(1/\epsilon)(\epsilon n)^2 \Es(\epsilon n) \\
&\qquad\qquad \leq C \log(1/\epsilon)\expect{M_{\epsilon n}}
\end{align*}
where the last inequality follows from Proposition~5.7 and Theorem~5.8 of~\cite{BM09}.
\end{proof}

\begin{theorem}\label{LERWupperboundThm2}
If $z \in \mathbb{D}$ and $\epsilon > 0$ is sufficiently small, then
\begin{equation}\label{LERWupperboundAsymp}
\lim_{n \to \infty} \condexpect{\nu_{Y_n}(\ball{z}{\epsilon})}{Y_n \cap \ball{z}{\epsilon} \neq \emptyset} \leq C \log(1/\epsilon) \epsilon^{5/4}.
\end{equation}
\end{theorem}

\begin{proof}
It follows from Proposition~6.2 of~\cite{BM09} that
$$
\frac{\expect{M_{\epsilon n}}}{\expect{M_n}} \asymp \frac{(\epsilon n)^2 \Es(\epsilon n)}{n^2 \Es(n)}$$
and from Theorem~3.9 of~\cite{BM09} that
$\Es(n) \asymp   \Es(\epsilon n) \Es(\epsilon n, n)$,
and so combining these two statements implies that
$$\frac{\expect{M_{\epsilon n}}}{\expect{M_n}} \asymp \frac{\epsilon^2}{\Es(\epsilon n, n)}.$$
Moreover, Theorem~3.9 of~\cite{BM09} also implies that there exists a $C<\infty$ such that for all $\epsilon>0$, there exists an $N=N(\eps)$ such that
$$ \Es(\epsilon n, n) \ge C^{-1}\epsilon^{3/4}$$
if $n \ge N$. Thus,
$$\frac{\expect{M_{\epsilon n}}}{\expect{M_n}} \le C\frac{\epsilon^2}{\epsilon^{3/4}} = C\epsilon^{5/4}$$
so using~\eqref{LERWupperboundstatement}, we conclude that~\eqref{LERWupperboundAsymp} follows.
\end{proof}

We will now discuss the lower bound in Conjecture~\ref{LERWconj1}. One approach to establishing this bound is to revisit the proofs of Theorem~\ref{LERWupperboundThm}
and Theorem~\ref{LERWupperboundThm2} to try and  see if up-to-constant estimates could be proved instead of just upper bounds. Indeed, the final step in a proof of a lower bound would be the same as the final step in the proof of the upper bound, except with inequalities reversed. This means that the proof would depend on how good a lower bound one could obtain for the following two quantities:
\begin{enumerate}
\item  $\condprob{w \in X^n}{X^n \cap \ball{z}{\epsilon/2} \neq \emptyset}$ for $w \in \ball{z}{\epsilon/2}$ and
\item $G^W(x,w)$.
\end{enumerate}

Concerning the first quantity, one of the key statements that we prove is~\eqref{s4eq1}. If we had the inequality reversed in~\eqref{s4eq1}, then we would have a good handle on
$\condprob{w \in X^n}{X^n \cap \ball{z}{\epsilon/2} \neq \emptyset}$.  Thus, we conjecture the following.

\begin{conjecture}\label{s4eq1lowerbound}
$\prob{\Lo(Z[0,\xi_x]) \cap W^w[1, \sigma_{\mathbb{D}}] = \emptyset} \asymp \Es(\epsilon n)$
\end{conjecture}

As already noted,  the upper bound in Conjecture~\ref{s4eq1lowerbound} is established in the proof of Theorem~\ref{LERWupperboundThm}. Therefore, we need to consider the lower bound.
Define $\beta^1$ and $\beta^2$ as in the proof of Theorem~\ref{LERWupperboundThm} so that
\begin{align*}
&\prob{\Lo(Z[0,\xi_x]) \cap W^w[1, \sigma_{\mathbb{D}}] = \emptyset}\\
&\qquad= \condprob{\Lo(Z[0,\xi_x]) \cap W^w[1, \sigma_{\mathbb{D}}] = \emptyset}{\beta^1 \cap \beta^2 = \emptyset} \prob{\beta^1 \cap \beta^2 = \emptyset}.
\end{align*}
As before,
\begin{equation*}
\prob{\beta^1 \cap \beta^2 = \emptyset} \asymp \Es(\epsilon n).
\end{equation*}
We now want to bound the term
\begin{equation*}
\condprob{\Lo(Z[0,\xi_x]) \cap W^w[1, \sigma_{\mathbb{D}}] = \emptyset}{\beta^1 \cap \beta^2 = \emptyset}
\end{equation*}
from below.
Let $y_i$ be the point where $\beta^i$ leaves $B(w, \epsilon / 8)$, $i=1,2$. By the domain Markov property (see Lemma~3.2 of~\cite{BM09}), conditioned on $\beta^1$, the rest of $\Lo(Z[0,\xi_x])$ is obtained by running $Z$ started at $y_1$ conditioned to hit $x$ before hitting $\beta^1$, i.e., by running a random walk started at $y_1$ conditioned to leave $\D \setminus (\beta^1 \cup \alpha)$ at $x$.
By a separation lemma (see Proposition~3.6 of~\cite{BM09}), there exists $c > 0$ such that conditioned on the event $\{ \beta^1 \cap \beta^2 = \emptyset \}$, $\beta^1 \cap B(y_2, c\epsilon) = \emptyset$ and $\beta^2 \cap B(y_1, c\epsilon) = \emptyset$.
Furthermore, by Proposition~2.3 of~\cite{BM09}, with probability bounded below by $c > 0$, one can move a distance $c \epsilon$ away from $B(y_1, \epsilon/8)$.
Therefore it suffices to find two disjoint sets $U_1$ and $U_2$ such that $U_i \cap B(w,\epsilon/8) \subset B(y_i, c\epsilon)$, $i=1$, $2$, and such that  with probability bounded below by $c>0$ we have the following.
\begin{enumerate}
\item A random walk started at $y_1$ conditioned to leave $\D \setminus (\beta^1 \cup \alpha)$ at $x$ stays in $U_1$. Note that  Proposition~2.3 of~\cite{BM09} will be useful for this.
\item A random walk started at $y_2$ conditioned to leave $\D$ before hitting $\alpha$ stays in $U_2$.
\end{enumerate}
This was done in Proposition~5.7 of~\cite{BM09} where $\alpha = \emptyset$ and $x = 0$, and in Lemma~6.1 of~\cite{BM09} where $\alpha$ was assumed to be contained in a square centred at the origin and $w$ was distance $\epsilon$ from the square. In these cases the two conditioned walks in 1 and 2 above were ``pushed'' in opposite directions (the first in towards the origin and the second out towards the boundary). In our case we do not have any a priori control over $\alpha$. The worst situation would seem to be where $\alpha$ winds around $z$ in which case the only way to escape $\alpha$ is to go towards its tip $x$ and therefore the two conditioned walks above are pulled in the same direction.

Concerning the second quantity,  the key statement is~\eqref{may24eq3} which gives a decomposition of  $G^W(x,w)$ in terms of three probabilities and the simple random walk Green's function.  Upper and lower bounds for the simple random walk Green's function term are given by~\eqref{may24eq2b} so what remains is to estimate the piece with the three probabilities. An upper bound is given by~\eqref{Mar21eq1} which, as explained in the proof of Theorem~\ref{LERWupperboundThm}, follows from the discrete Harnack principle. Unfortunately, the best that we are able to achieve concerning the lower bound is the following. Since
\begin{equation*}
\probb{y}{\xi_w < \xi_\alpha \wedge \sigma_{\D}} \geq \probb{y}{\xi_w < \sigma_{B(w,\epsilon/4)}} \geq \frac{c}{\log \epsilon n}
\end{equation*}
we conclude
\begin{equation} \label{s4eq2asymp}
\frac{c}{\log \epsilon n} \leq \frac{\probb{x}{\xi_w < \xi_\alpha \wedge \sigma_{\mathbb{D}}}}{\probb{x}{\xi_{B_{\epsilon/8}(w)} < \xi_\alpha \wedge \sigma_{\D}}} \leq C \frac{\log(1/\epsilon)}{\log n}.
\end{equation}
In order to finish the lower bound on  $G^W(x,w)$ we would need an analogue of Corollary~4.5 of~\cite{BM09}; as stated, that result will not be true for those $\alpha$ that wind around $w$.

While it is reasonable to expect that the lower bound in Conjecture~\ref{LERWconj1} might ultimately be established using a modification of the strategy presented here, it is not at all clear that an up-to-constants result can be improved to an asymptotic result using existing machinery. Still, we think it is reasonable to expect that the following conjecture might one day be proved.

\begin{conjecture}\label{LERWconj2}
If $z \in \mathbb{D}$ and $\epsilon > 0$ is sufficiently small, then
\begin{equation*}
\condexpect{\nu_{Y_n}(\ball{z}{\epsilon})}{Y_n \cap \ball{z}{\epsilon} \neq \emptyset} = \frac{\expect{M_{\epsilon n}}}{\expect{M_n}} + o(1).
\end{equation*}
\end{conjecture}

\section{Conclusion}\label{conclsect}

We have outlined a strategy for showing convergence of loop-erased random walk to SLE(2) in the natural time parametrization. The approach is to study the random spatial occupation measure induced by the discrete curves and show that they converge in law to the natural occupation measure for SLE(2) curves that was recently introduced by Lawler and Sheffield~\cite{LS09}. There are two main steps: first prove tightness, and then show that all subsequential limits share certain characterizing properties. The most important of these are conformal covariance and the domain Markov property. We observed that for the parametrization of the discrete curves by $\sigma_n(t) = \expect{M_n}t$ one has tightness as a consequence of Barlow and Masson~\cite{BM09}. For the more desirable choice $\sigma_n(t) = n^{5/4}t$, the tightness remains an open issue. To prove conformal covariance and the domain Markov property the most important step is to show convergence of the expected occupation measure to the SLE Green's function on the given domain. This is a challenging problem. Finally, we presented some rigorous estimates for loop-erased random walk that strongly suggest this convergence takes place on the disk.

We mention one last idea that came up in conversations with Ed Perkins. We mostly focussed our attention on the choice of speed function $\sigma_n(t) = c_n t$ with $c_n$ determined only by the lattice spacing. This is natural from the point of view of discrete lattice models, but to simplify the proof of conformal covariance it might be fruitful to make the choice of $c_n$ domain dependent. In particular, one could construct the discrete time parametrization in a ``domain-adapted'' way by making the time duration of each step depend on the remaining domain that the walk can move through.

\subsection*{Acknowledgements}

The authors would like to express their gratitude to the Banff International Research Station for Mathematical Innovation and Discovery (BIRS), the Mathematical Sciences Research Institute (MSRI), and the Simons Center for Geometry and Physics where much of this work was carried out. In particular, the authors benefitted from participating in a \emph{Research in Teams} at BIRS, as well as the  \emph{Program on Random Spatial Processes} at MSRI, and the \emph{Program on Conformal Geometry} at the Simons Center. Thanks are owed to Ed Perkins, Martin Barlow, and Greg Lawler for useful discussions.
The research of the first two authors was supported in part by the Natural Sciences and Engineering Research Council (NSERC) of Canada.

\end{document}